\documentclass[10pt]{article}
\usepackage{amsmath,amsthm,amsfonts,latexsym,amsopn,verbatim,amscd,amssymb}
\usepackage{tkz-graph}
\usepackage{pgfplots}
\usepackage{caption}
\theoremstyle{plain}
\newtheorem{theorem}{Theorem}[section]
\newtheorem{corollary}[theorem]{Corollary}

\newtheorem{proposition}[theorem]{Proposition}

\newcommand{\bnum}{\begin{enumerate}}
\newcommand{\enum}{\end{enumerate}}

\numberwithin{equation}{section}

\DeclareMathOperator{\diam}{diam}

\begin{document}
\title{\textbf{On $r$-noncommuting graph of  finite rings}}
\author{Rajat Kanti Nath$^{1,*}$, Monalisha Sharma$^1$, Parama Dutta$^2$ and \\ Yilun Shang$^{3,*}$}
\date{}
\maketitle
\vspace{-1cm}
\begin{center}\small{\it
$^1$Department of Mathematical Sciences, Tezpur University, Napaam-784028, Sonitpur, Assam, India;  rajatkantinath@yahoo.com (R.K.N); monalishasharma2013@gmail.com (M.S.) 

$^2$Department of Mathematics,   Lakhimpur Girls’ College, Lakhimpur 787031, \\ Assam, India;  parama@gonitsora.com (P.D.)

$^{3}$ Department of Computer and Information Sciences, Northumbria
University, Newcastle NE1 8ST, UK; 
yilun.shang@northumbria.ac.uk (Y.S.)

\thanks{$^*$Corresponding Authors}
}
\end{center}

\begin{abstract}
Let $R$ be a finite ring and $r\in R$. The $r$-noncommuting graph of $R$, denoted by $\Gamma_R^r$,  is a simple undirected graph whose vertex set is $R$ and two vertices $x$ and $y$ are adjacent if and only if $[x,y] \neq r$ and $-r$. In this paper, we study several properties of $\Gamma_R^r$. We show that $\Gamma_R^r$ is not  a regular graph, a lollipop graph and complete bipartite graph. Further, we consider an induced subgraph of $\Gamma_R^r$ (induced by the non-central elements of $R$) and obtained some characterizations of $R$.
\end{abstract}



\section{Introduction}
Throughout the paper $R$ denotes a finite ring and $r\in R$. Let $Z(R):= \{z\in R:zr = rz \text{ for all } r\in R\}$ be the center of $R$. For any element $x \in R$, the centralizer of $x$ in $R$ is a subring given by $C_R(x) : = \{y \in R : xy = yx\}$. Clearly, $Z(R) = \underset{x \in R}{\cap}C_R(x)$. For any two elements $x$ and $y$ of $R$, $[x,y] := xy - yx$ is called the additive commutator of $x$ and $y$. Let $K(R) = \{[x,y]: x,y\in R\}$ and   $[R, R]$ and $[x, R]$ for $x \in R$  denote the additive subgroups of $(R, +)$ generated by the sets $K(R)$ and $\lbrace [x, y] : y\in R\rbrace$ respectively.

The study of graphs defined on algebraic structures have been an active topic of research in the last few decades. Recently, Erfanian, Khashyarmanesh and Nafar \cite{erfanian1} considered noncommuting graphs of finite rings. Recall that the noncommuting graph of a finite noncommutative ring $R$ is a simple undirected graph whose vertex set is $R \setminus Z(R)$ and two vertices $x$ and $y$ are adjacent if and only if $xy\neq yx$. A generalization of this graph can be found in \cite{jutirekha2}. 
The complement of noncommuting graph, called commuting graph,  of a finite noncommutative ring is considered in   \cite{dutta-nath} and \cite{o&v}.

 In this paper, we introduce and study $r$-noncommuting graph of a finite ring $R$ for any given element $r \in R$. The  $r$-noncommuting graph of $R$, denoted by $\Gamma_R^r$, is a simple undirected graph whose vertex set is $R$ and two vertices $x$ and $y$ are adjacent if and only if $[x,y] \neq r$ and $-r$. Clearly, $\Gamma_R^r = \Gamma_R^{-r}$. If $r = 0$ then the induced subgraph of $\Gamma_R^r$  with vertex set $R \setminus Z(R)$, denoted by $\Delta^r_R$, is nothing but the noncommuting graph of $R$.   Note that $\Gamma_R^r$ is $0$-regular graph if $r = 0$ and $R$ is  commutative. Also,  $\Gamma_R^r$ is complete if $r \notin K(R)$. Thus for $r \notin K(R)$,  $\Gamma_R^r$ is $n$-regular if and only if $R$ is  of order $n + 1$. Therefore throughout the paper we shall consider $r \in K(R)$.   The motivation of this paper lies in \cite{nEGJ-16,nEGJ-17,nEM-18,tolue} where analogous notion of this graph is studied in case of finite groups.

 In Section 2, we first compute the degree of any vertex of $\Gamma_R^r$ in terms of its centralizers. Then we characterize $R$ if   $\Gamma_R^r$ is a tree or a star graph. We further show that   $\Gamma_R^r$ is not a  regular graph (if $r\in K(R)$), a lollipop graph and complete bipartite graph for any noncommutative ring $R$. In Section 3, we show that $\Gamma_{R_1}^{r}$ is isomorphic to $\Gamma_{R_2}^{\psi(r)}$ if $(\phi,\psi)$ is an isoclinism between two finite rings $R_1$ and $R_2$   such that $|Z(R_1)| = |Z(R_2)|$. In Section 4, we consider $\Delta^r_R$ and obtain some characterization of $R$ along with other results. As a consequence of our results we determine some positive integers $n$ such that the noncommuting graph of $R$  is $n$-regular and give some characterizations of such rings.

 It was shown in \cite{fine93} that there are only two noncommutative rings (up to isomorphism) having order $p^2$, where $p$ is a prime, and the rings are given by
\[
E(p^2) = \langle a, b : pa = pb = 0, a^2 = a, b^2 = b, ab = a, ba = b \rangle
\]
\[
\text{and } F(p^2) = \langle x, y : px = py = 0, x^2 = x, y^2 = y, xy = y, yx = x \rangle.
\]
Following figures show the graphs $\Gamma_{E(p^2)}^r$ for $p = 2, 3$.

\begin{minipage}[t]{.005\linewidth}
\begin{tikzpicture}[scale=0.25]
\renewcommand*{\VertexSmallMinSize}{1 pt}
\GraphInit[vstyle=Welsh]
\Vertices[unit=3]{circle}{$a$,$b$,$a+b$,$0$}
\AddVertexColor{black}{$a$,$b$,$a+b$,$0$}

\Edges($a$,$b$)
\Edges($a$,$a+b$)

\Edges($b$,$a+b$)

\end{tikzpicture}

\end{minipage}
\hspace{6.5cm}
\begin{minipage}[t]{.005\linewidth}
\begin{tikzpicture}[scale=0.25]
\renewcommand*{\VertexSmallMinSize}{1 pt}
\GraphInit[vstyle=Welsh]
\Vertices[unit=3]{circle}{$a$,$b$,$a+b$,$0$}
\AddVertexColor{black}{$a$,$b$,$a+b$,$0$}

\Edges($0$,$a$)
\Edges($0$,$a+b$)

\Edges($b$,$0$)
\end{tikzpicture}
\end{minipage}
\begin{center}
\small{Figure 1: $\Gamma_{E(4)}^0$ \hspace{4.5cm}   Figure 2: $\Gamma_{E(4)}^{a+b}$}
\end{center}

\begin{minipage}[t]{.005\linewidth}

\begin{tikzpicture}[scale=0.4]
\renewcommand*{\VertexSmallMinSize}{.02 mm}
\GraphInit[vstyle=Welsh]
\Vertices[unit=3]{circle}{$2a$,$b$,$2b$,$a+b$,$2a+2b$,$a+2b$,$2a+b$,$0$,$a$}
\AddVertexColor{black}{$2a$,$b$,$2b$,$a+b$,$2a+2b$,$a+2b$,$2a+b$,$0$,$a$}

\Edges($a$,$b$)
\Edges($a$,$2b$)
\Edges($a$,$a+b$)
\Edges($a$,$2a+2b$)
\Edges($a$,$a+2b$)
\Edges($a$,$2a+b$)
\Edges($2a$,$b$)
\Edges($2a$,$2b$)
\Edges($2a$,$a+b$)
\Edges($2a$,$2a+2b$)
\Edges($2a$,$a+2b$)
\Edges($2a$,$2a+b$)
\Edges($b$,$a+b$)
\Edges($b$,$2a+2b$)
\Edges($b$,$a+2b$)
\Edges($b$,$2a+b$)
\Edges($2b$,$a+b$)
\Edges($2b$,$2a+2b$)
\Edges($2b$,$a+2b$)
\Edges($2b$,$2a+b$)
\Edges($a+b$,$a+2b$)
\Edges($a+b$,$2a+b$)
\Edges($2a+2b$,$a+2b$)
\Edges($2a+2b$,$2a+b$)
\end{tikzpicture}

\end{minipage}
\hspace{6.5cm}
\begin{minipage}[t]{.005\linewidth}
\begin{tikzpicture}[scale=0.4]
\renewcommand*{\VertexSmallMinSize}{1 pt}
\GraphInit[vstyle=Welsh]
\Vertices[unit=3]{circle}{$2a$,$b$,$2b$,$a+b$,$2a+2b$,$a+2b$,$2a+b$,$0$,$a$}
\AddVertexColor{black}{$2a$,$b$,$2b$,$a+b$,$2a+2b$,$a+2b$,$2a+b$,$0$,$a$}
\Edges($a$,$2a$)
\Edges($a$,$0$)
\Edges($2a$,$0$)
\Edges($b$,$0$)
\Edges($b$,$2b$)
\Edges($2b$,$0$)
\Edges($a+b$,$0$)
\Edges($a+b$,$2a+2b$)
\Edges($2a+2b$,$0$)
\Edges($a+2b$,$2a+b$)
\Edges($a+2b$,$0$)
\Edges($2a+b$,$0$)
\end{tikzpicture}
\end{minipage}
\begin{center}
\hspace{1.5 cm} \small{Figure 3: $\Gamma_{E(9)}^0$ \hspace{3.8cm}   Figure 4: $\Gamma_{E(9)}^{a+2b} = \Gamma_{E(9)}^{2a+b}$}
\end{center}

It may be noted here that the graphs $\Gamma_{F(4)}^0, \Gamma_{F(4)}^{x+y}, \Gamma_{F(9)}^0$ and $\Gamma_{F(9)}^{x+2y}$ are isomorphic to $\Gamma_{E(4)}^0, \Gamma_{E(4)}^{a+b}, \Gamma_{E(9)}^0$ and $\Gamma_{E(9)}^{a+2b}$ respectively.

\section{Some properties}
In this section, we characterize $R$ if   $\Gamma_R^r$ is a tree or a star graph. We also show the non-existence of finite noncommutative rings $R$ whose $r$-noncommuting graph is a   regular graph (if $r\in K(R)$), a lollipop graph or a complete bipartite graph. However, we first compute degree of any vertices in the graph $\Gamma_R^r$.  For any two given elements $x$ and $r$ of $R$, we write $T_{x, r}$ to denote the generalized centralizer $\{y\in R : [x, y] = r\}$ of $x$. 
The following proposition gives  degree of any vertices  of $\Gamma_R^r$ in terms of its  generalized centralizers.
\begin{proposition}\label{deg_prop_1}
Let $x$ be any vertex in  $\Gamma_R^r$.  Then
\begin{enumerate}
\item  $\deg(x)= |R| - |C_R(x)|$ if $r=0$.
\item if $r \neq 0$ then
    $\deg(x) = \begin{cases}
       |R| - |T_{x, r}| - 1, & \mbox{if $2r= 0$} \\
       |R| - 2|T_{x, r}| -  1, & \mbox{if $2r \neq 0$}.
     \end{cases} $
\end{enumerate}
\end{proposition}

\begin{proof}
(a) If $r = 0$  then $\deg(x)$ is the number of $y \in R$ such that $xy \ne yx$. Note that $|C_R(x)|$ gives the number of elements that commute with $x$. Hence, $\deg(x)= |R| - |C_R(x)|$.

(b) Consider the case when  $r \neq 0$. If  $2r = 0$ then $r = -r$. Note that $y \in R$ is not adjacent to $x$  if and only if $y = x$ or $y \in T_{x, r}$. Therefore, $\deg(x) = |R| - |T_{x, r}| - 1$. If  $2r \neq 0$ then $r \neq -r$. It is easy to see that $T_{x, r} \cap T_{x, -r} = \emptyset$ and  $y \in T_{x, r}$ if and only if $-y \in T_{x, -r}$. Therefore, $|T_{x, r}| = |T_{x, -r}|$. Note that $y \in R$ is not adjacent to $x$  if and only if $y = x$ or $y \in T_{x, r}$ or $y \in T_{x, -r}$. Therefore, $\deg(x) = |R| - |T_{x, r}| - |T_{x, -r}| - 1$. Hence the result follows.
\end{proof}
 The following corollary  gives  degree of any vertices  of $\Gamma_R^r$ in terms of its centralizers.
\begin{corollary}\label{deg_cor_1}
Let $x$ be any vertex in  $\Gamma_R^r$.   
\begin{enumerate}
\item If $r \neq 0$ and  $2r= 0$ then
$\deg(x) = \begin{cases}
           |R| - 1, & \mbox{if $T_{x, r} = \emptyset$}  \\
           |R| - |C_R(x)| - 1, & \mbox{otherwise}.
\end{cases}$
\item If $r \neq 0$ and  $2r \neq 0$ then
$\deg(x) = \begin{cases}
           |R| - 1, & \mbox{if $T_{x, r} = \emptyset$}  \\
           |R| - 2|C_R(x)| - 1, & \mbox{otherwise}.
\end{cases}$
\end{enumerate}
\end{corollary}
\begin{proof}
Notice that $T_{x, r} \ne \emptyset$ if and only if $r \in [x, R]$. Suppose that $T_{x, r}\neq \emptyset$.  Let $t \in T_{x, r}$ and  $p \in t + C_R(x)$. Then $[x, p] = r$ and so $p \in T_{x, r}$. Therefore,  $t + C_R(x) \subseteq T_{x, r}$. Again, if $y \in T_{x, r}$ then  $(y - t) \in C_R(x)$ and so $y \in t + C_R(x)$. Therefore, $T_{x, r}\subseteq t + C_R(x)$. Thus $|T_{x, r}| = |C_R(x)|$ if $T_{x, r}\neq \emptyset$. Hence the result follows from  Proposition \ref{deg_prop_1}.
\end{proof}




\begin{proposition}
  Let $R$ be a ring with unity. The $r$-noncommuting graph $\Gamma_R^r$ is  a tree if and only if $|R| = 2$ and $r\neq 0$.
\end{proposition}
\begin{proof}
If $r = 0$ then, by Proposition \ref{deg_prop_1}(a), we have $\deg(r) = 0$. Hence, $\Gamma_R^r$ is not a tree. Suppose that $r\neq 0$. If $R$ is commutative then $r \notin K(R)$. Hence, $\Gamma_R^r$ is  complete graph. Therefore $\Gamma_R^r$ is a tree if and only if $|R|=2$. If $R$ is noncommutative  then $[x, 0] \neq r, - r$ and  $[x, 1] \neq r, - r$ for any $x \in R$. Therefore $\deg(x) \geq 2$ for all $x \in R$.  Hence, $\Gamma_R^r$ is not a tree.
\end{proof}

\begin{proposition}\label{lollipop}
If $R$ is non-commutative then $\Gamma_{R}^{r}$ is not a lollipop graph.
\end{proposition}
\begin{proof}
Let $\Gamma_{R}^{r}$ be a lollipop graph. Then there exists an element $x\in R$ such that $\deg(x) = 1$. If $r = 0$ then $x \notin Z(R)$ and so $|C_R(x)| \leq \frac{|R|}{2}$. Also, by Proposition \ref{deg_prop_1}(a), we have $\deg(x) = |R|-|C_R(x)|$. These give  $|R|-|C_R(x)| = 1$. Hence, $|R| \leq 2$,  a contradiction.

 If $r \neq 0$  then, by Corollary \ref{deg_cor_1}, we have  $\deg(x) = |R| - 1, |R| - |C_R(x)| - 1$ or $|R| - 2|C_R(x)| - 1$. These give $|R| - |C_R(x)|  = 2$ or $|R| - 2|C_R(x)|  = 2$. Clearly $x \notin Z(G)$ and so $|C_R(x)| \leq \frac{|R|}{2}$. Therefore, if $|R| - |C_R(x)|  = 2$ then $|R| \leq 4$. If $|R| - 2|C_R(x)|  = 2$ then $|R|$ is even and  $|C_R(x)| \leq \frac{|R|}{2}$. Therefore,  $|R| \leq 6$. Since $R$ is noncommutative we have $|R| = 4$.
Therefore, if $r \neq 0$ then  $\Gamma_{R}^{r}$ is isomorphic to a star graph on $4$ vertices. Hence, $\Gamma_{R}^{r}$ is not a lollipop graph.
\end{proof}
The proof of Proposition \ref{lollipop} also ensures the following result.
\begin{proposition}
Let  $R$ be a noncommutative ring. Then $\Gamma_{R}^{r}$ is  a star graph if and only if $R$ is isomorphic to $
E(4) = \langle a, b : 2a = 2b = 0, a^2 = a, b^2 = b, ab = a, ba = b \rangle
$
or
$
F(4) = \langle a, b : 2a = 2b = 0, a^2 = a, b^2 = b, ab = b, ba = a \rangle
$.
\end{proposition}
\noindent In fact, if $R$ is noncommutative having more than four elements  then  there is no vertex of degree one in $\Gamma_{R}^{r}$.

It is observed that $\Gamma_{R}^{r}$ is $(|R| - 1)$-regular if $r \notin K(R)$. Also, if $r = 0$ and $R$ is commutative then $\Gamma_{R}^{r}$ is $0$-regular.    In the following proposition, we show that $\Gamma_{R}^{r}$ is not regular if $r \in K(R)$.

\begin{proposition}
Let $R$ be a noncommutative ring and $r\in K(R)$. Then  $\Gamma_{R}^{r}$ is not regular.
\end{proposition}
\begin{proof}
If $r = 0$ then, by Proposition \ref{deg_prop_1}(a), we have $\deg(r) = 0$. Let $x \in R$ be a non-central element. Then $|C_R(x)| \neq |R|$. Therefore, by Proposition \ref{deg_prop_1}(a), $\deg(x) \neq 0 = \deg(r)$. This shows that $\Gamma_{R}^{r}$ is not regular.  If $r \neq 0$ then $T_{0,r} = \emptyset$. Therefore, by Corollary \ref{deg_cor_1}, we have $\deg(0) = |R| - 1$. Since $r\in K(R)$, there exists $0 \neq x \in R$ such that $T_{x,r} \neq \emptyset$. Therefore, by Corollary \ref{deg_cor_1}, we have $\deg(x) = |R| - |C_R(x)| - 1$ or $|R| - 2|C_R(x)| - 1$. If $\Gamma_{R}^{r}$ is regular then $\deg(x) = \deg(0)$. Therefore
\[
|R| - |C_R(x)| - 1 = |R| - 2|C_R(x)| - 1 = |R| -  1
\]
which gives $|C_R(x)| = 0$, a contradiction. Hence, $\Gamma_{R}^{r}$ is not regular. This completes the proof.
\end{proof}

We conclude this section with the following result.
\begin{proposition}
Let $R$ be a finite ring.
\begin{enumerate}
\item If $r=0$ then $\Gamma_{R}^{r}$ is  not complete bipartite.
\item If $r\neq 0$ then $\Gamma_{R}^{r}$ is  not complete bipartite for $|R|\geq 3$ with $|Z(R)|\geq 2$.
\end{enumerate}
\end{proposition}
\begin{proof}
    Let $\Gamma_{R}^{r}$ be complete bipartite. Then there exist subsets $V_1$ and $V_2$ of $R$ such that $V_1 \cap V_2 = \emptyset, V_1 \cup V_2 = R$  and   if $x\in V_1$ and $y\in V_2$ then $x$ and $y$ are adjacent.

(a) If $r = 0$ then for $x \in V_1$ and $y \in V_2$ we have $[x, y]\neq 0$. Therefore, $[x, x+y]\neq 0$ which implies $x+y \in V_2$. Again $[y, x+y]\neq 0$ which implies $x+y \in V_1$. Thus $x + y \in V_1 \cap V_2$, a contradiction. Hence $\Gamma_{R}^{r}$ is  not complete bipartite.

(b) If $r\neq 0, |R|\geq 3$ and $|Z(R)|\geq 2$ then for any $z_1,z_2\in Z(R)$, $z_1$ and $z_2$ are adjacent. Let us take $z_1\in V_1$ and $z_2\in V_2$. Since  $|R|\geq 3$ we have $x\in R$ such that $x\neq z_1$ and $x\neq z_2$.  Also $[x, z_1] = 0 = [x, z_2]$. Therefore $x$ is adjacent to both $z_1$ and $z_2$. Therefore $x \notin V_1\cup V_2 = R$, a contradiction. Hence $\Gamma_{R}^{r}$ is  not complete bipartite.
\end{proof}

\section{Isoclinism between rings and $\Gamma_{R}^{r}$}
In 1940, Hall \cite{pH40} introduced  isoclinism between two groups. Recently, Buckley et al. \cite{BMS} have introduced isoclinism between two rings. Let $R_1$ and $R_2$ be two rings. A pair of additive group  isomorphisms $(\phi, \psi)$ where $\phi : \frac{R_1}{Z(R_1)} \to \frac{R_2}{Z(R_2)}$   and $\psi : [R_1, R_1] \to [R_2, R_2]$ is called an isoclinism between $R_1$ to $R_2$ if  $\psi([u, v]) = [u', v']$ whenever $\phi(u + Z(R_1)) = u' + Z(R_2)$ and $\phi(v + Z(R_1)) = v' + Z(R_2)$. Two rings are called isoclinic if there exists an isoclinism between them. If $R_1$ and $R_2$ are two isomorphic rings and $\alpha: R_1 \to R_2$ is an isomorphism then  it is easy to see that $\Gamma_{R_1}^{r} \cong \Gamma_{R_2}^{\alpha(r)}$. In the following proposition we show that $\Gamma_{R_1}^{r} \cong \Gamma_{R_2}^{\psi(r)}$ if      $R_1$ and $R_2$ are two isoclinic rings with isoclinism  $(\phi,\psi)$.
\begin{proposition}
  Let $R_1$ and $R_2$ be two finite rings such that $|Z(R_1)| = |Z(R_2)|$. If $(\phi,\psi)$ is an isoclinism between $R_1$ and $R_2$ then
  \[
  \Gamma_{R_1}^{r} \cong \Gamma_{R_2}^{\psi(r)}.
  \]
\end{proposition}
\begin{proof}
  Since $\phi : \frac{R_1}{Z(R_1)} \to \frac{R_2}{Z(R_2)}$ is an isomorphism,  $\frac{R_1}{Z(R_1)}$ and $\frac{R_2}{Z(R_2)}$ have same number of elements. Let  $\left|\frac{R_1}{Z(R_1)}\right| = \left|\frac{R_2}{Z(R_2)}\right| = n$. Again since $|Z(R_1)|=|Z(R_2)|$, there exists a bijection $\theta:Z(R_1)\to Z(R_2)$. Let $\{r_i : 1\leq i \leq n\}$ and $\{s_j : 1\leq j\leq n\}$ be two transversals of $\frac{R_1}{Z(R_1)}$ and $\frac{R_2}{Z(R_2)}$ respectively. Let $\phi :\frac{R_1}{Z(R_1)} \to \frac{R_2}{Z(R_2)}$ and $\psi:[R_1,R_1] \to [R_2,R_2]$ be defined as $\phi(r_i + Z(R_1)) = s_i + Z(R_2)$ and $\psi([r_i+z_1,r_j+z_2])=[s_i + z_1', s_j+ z_2']$ for some $z_1,z_2\in Z(R_1)$, $z_1', z_2' \in Z(R_2)$ and $1\leq i,j\leq n$.

Let us define a map $\alpha : R_1 \to R_2$ such that $\alpha (r_i + z) = s_i + \theta(z)$ for $z \in Z(R)$. Clearly $\alpha$ is a bijection. We claim that $\alpha$ preserves adjacency. Let $x$ and  $y$ be two elements of $R_1$ such that $x$ and $y$ are adjacent. Then $[x, y] \neq r, -r$. We have $x = r_i + z_i$ and $y = r_j + z_j$ where  $z_i, z_j \in Z(R_1)$ and $1\leq i, j\leq n$. Therefore 
\begin{align*}
&[r_i + z_i, r_j + z_j] \ne r, -r\\
\Rightarrow & \psi([r_i + z_i, r_j + z_j]) \ne \psi(r), -\psi(r)\\
\Rightarrow & [s_i + \theta(z_i), s_j + \theta(z_j)] \ne \psi(r), -\psi(r)\\
\Rightarrow & [\alpha(r_i + z_i),\alpha(r_j + z_j)]  \ne \psi(r), -\psi(r)\\
\Rightarrow & [\alpha(x),\alpha(y)]  \ne \psi(r), -\psi(r).
\end{align*}
This shows that $\alpha(x)$ and $\alpha(y)$ are adjacent. Hence the result follows.
\end{proof}

\section{An induced subgraph}
We write $\Delta_R^r$ to denote the induced subgraph of $\Gamma_{R}^{r}$ with vertex set $R \setminus Z(R)$. It is worth mentioning that $\Delta_R^0$ is the noncommuting graph of $R$. If $r \ne 0$ then it is easy to see that the commuting graph of $R$ is a spanning subgraph of $\Delta_R^r$. The following result gives a condition such that $\Delta_R^r$ is  the commuting graph of $R$.
\begin{proposition} 
Let $R$ be a noncommutative ring and $r \ne 0$. If $K(R) = \{0, r, -r\}$ then $\Delta_R^r$ is the commuting graph of $R$.
\end{proposition}
\begin{proof}
The result follows from the fact that two vertices $x, y$ in $\Delta_R^r$ are adjacent if and only if $xy = yx$.   
\end{proof}
Let $\omega(\Delta_R^r)$ be the clique number of $\Delta_R^r$. The following result gives a lower bound for $\omega(\Delta_R^r)$.
\begin{proposition} 
Let $R$ be a noncommutative ring and $r \ne 0$. If $S$ is a commutative subring of $R$ having maximal order then    $\omega(\Delta_R^r) \geq |S| - |S\cap Z(R)|$. 
\end{proposition}
\begin{proof}
The result follows from the fact that the subset $S\setminus S\cap Z(R)$  of $R \setminus Z(R)$ is a clique of $\Delta_R^r$.
\end{proof}

By \cite[Theorem 2.1]{erfanian1}, it follows that the diameter of $\Delta_R^0$ is less than or equal to $2$. The next  result  gives some information regarding diameter of $\Delta_R^r$ when $r \ne 0$. We write $\diam(\Delta_R^r)$ and $d(x, y)$ to denote the diameter of $\Delta_R^r$ and  the distance between $x$ and $y$ in $\Delta_R^r$ respectively. For any two vertices $x$ and $y$, we write $x \sim y$ to denote  $x$ and $y$ are adjacent, otherwise $x \nsim y$.

\begin{theorem}
Let $R$ be a noncommutative ring and $r \in R \setminus Z(R)$ such that $2r \ne 0$.
\begin{enumerate}
\item If   $3r \ne 0$ then $\diam(\Delta_R^r) \leq 3$.
\item If $|Z(R)| = 1$, $|C_R(r)| \ne 3$ and $3r = 0$ then $\diam(\Delta_R^r) \leq 3$.
\end{enumerate} 
\end{theorem}
\begin{proof}

(a) 
If $x \sim r$ for all $x \in R \setminus Z(R)$ such that $x \ne r$ then, it is easy to see that $\diam(\Delta_R^r) \leq 2$. Suppose  there exists a vertex $x \in R \setminus Z(R)$  such that $x \nsim r$. Then $[x, r] = r$ or $-r$. We have
\[
[x, 2r] = 2[x, r] = \begin{cases}
\,\,\,\,\,2r, & \text{ if }  [x, r] = r\\
-2r, & \text{ if }  [x, r] = -r.
\end{cases}
\]
Since $2r \ne 0$ we have $[x, 2r] \ne 0$ and hence $2r \in R \setminus Z(R)$. Also, $2r \ne r, -r$. Therefore, $[x, 2r] \ne r, -r$ and so $x \sim 2r$. Let $y \in R \setminus Z(R)$ such that $y \ne x$. If  $y \sim r$ then $d(x, y) \leq 3$ noting that $r \sim 2r$. If  $y \nsim r$ then   $y \sim  2r$ (as shown above). In this case  $d(x, y) \leq 2$.  Hence, $\diam(\Delta_R^r) \leq 3$. 

\vspace{.5cm}

(b) If $x \sim r$ for all $x \in R \setminus Z(R)$ such that $x \ne r$ then, it is easy to see that $\diam(\Delta_R^r) \leq 2$. Suppose  there exists a vertex $x \in R \setminus Z(R)$  such that $x \nsim r$. Let $y \in R \setminus Z(R)$ such that $y \ne x$. We consider the following two cases.

\noindent \textbf{Case 1:} $x \nsim r$ and $x \sim 2r$.

 If  $y \sim r$ then $d(x, y) \leq 3$ noting that $r \sim 2r$. Therefore, $\diam(\Delta_R^r) \leq 3$. 
 If  $y \nsim r$ but   $y \sim 2r$ then  $d(x, y) \leq 2$.
Consider the case when $y \nsim r$ as well as  $y \nsim 2r$. Therefore $[y, r] = r$ or $-r$. If $[y, r] = r$ then $[y, 2r] = 2[y, r] = 2r  = -r$; otherwise $y \sim 2r$, a contradiction.  Let $a \in C_R(r)$ such that $a \ne 0, r, -r$ (such element exists, since $|C_R(r)| > 3$). Clearly $a \in R \setminus Z(R)$. Suppose $y \sim a$. Then $x \sim 2r \sim a \sim y$ and so
$d(x, y) \leq 3$.
Suppose $y \nsim a$. Then $[y, a] = r$ or $-r$. If $[y, a] = r$ then  
\[
[y, r - a] = [y, r] - [y, a] = r - r =  0.
\]  
Note that $r - a \in R \setminus Z(R)$; otherwise $a = r$, a contradiction. Therefore, $y \sim r - a$. Also,
\[
[r - a, 2r] =   2[r, a] = 0.
\]
That is, $r - a \sim 2r$. Thus $x \sim   2r \sim r - a \sim y$ . Therefore, $d(x, y) \leq 3$.
If $[y, a] = -r$ then  
\[
[y, 2r - a] = [y, 2r] - [y, a] = -r - (-r) =  0.
\]  
Note that $2r - a \in R \setminus Z(R)$; otherwise $a = 2r = -r$, a contradiction. Therefore, $y \sim 2r - a$. Also,
\[
[2r - a, 2r] =   2[r, a] = 0.
\]
That is, $2r - a \sim 2r$. Thus $x \sim   2r \sim 2r - a \sim y$. Therefore, $d(x, y) \leq 3$.

 If $[y, r] = -r$ then $[y, 2r] = 2[y, r] = -2r  = r$; otherwise $y \sim 2r$, a contradiction.  Let $a \in C_R(r)$ such that $a \ne 0, r, -r$.  Suppose $y \sim a$. Then $x \sim 2r \sim a \sim y$ and so $d(x, y) \leq 3$. 
  Suppose $y \nsim a$. Then $[y, a] = r$ or $-r$. If $[y, a] = r$ then  
\[
[y, r + a] = [y, r] + [y, a] = -r + r =  0.
\]  
Note that $r + a \in R \setminus Z(R)$; otherwise $a = -r$, a contradiction. Therefore, $y \sim r + a$. Also,
\[
[r + a, 2r] =   2[a, r] = 0.
\]
That is, $r + a \sim 2r$. Thus $x \sim   2r \sim r + a \sim y$ . Therefore, $d(x, y) \leq 3$.
If $[y, a] = -r$ then  
\[
[y, 2r + a] = [y, 2r] + [y, a] = r + (-r) =  0.
\]  
Note that $2r + a \in R \setminus Z(R)$; otherwise $a = -2r = r$, a contradiction. Therefore, $y \sim 2r + a$. Also,
\[
[2r + a, 2r] =   2[a, r] = 0.
\]
That is, $2r + a \sim 2r$. Thus $x \sim   2r \sim 2r + a \sim y$ . Therefore, $d(x, y) \leq 3$ and hence $\diam(\Delta_R^r) \leq 3$.

\noindent \textbf{Case 2:} $x \nsim r$ and $x \nsim 2r$.

Let $a \in C_R(r)$ such that $a \ne 0, r, -r$. 

\noindent \textbf{Subcase 2.1:}  $x \sim a$

If  $y \sim r$ then $y \sim r \sim a \sim x$. Therefore $d(x, y) \leq 3$. 
 If  $y \nsim r$ but   $y \sim 2r$ then $y \sim 2r \sim a \sim x$. Therefore,  $d(x, y) \leq 3$.
Consider the case when $y \nsim r$ as well as  $y \nsim 2r$. Therefore $[y, r] = r$ or $-r$. If $[y, r] = r$ then $[y, 2r] = 2[y, r] = 2r  = -r$; otherwise $y \sim 2r$, a contradiction.  
 Suppose $y \sim a$. Then $y \sim a \sim x $ and so
$d(x, y) \leq 2$.
Suppose $y \nsim a$. Then $[y, a] = r$ or $-r$. If $[y, a] = r$ then  
$[y, r - a] =   0$. 
 Therefore, $y \sim r - a \sim a \sim x$. 
Therefore, $d(x, y) \leq 3$.
If $[y, a] = -r$ then  $[y, 2r - a] =  0$.  Therefore, $y \sim 2r - a \sim a \sim x$ and so $d(x, y) \leq 3$.

If $[y, r] = -r$ then  $[y, 2r] = 2[y, r] = -2r  = r$; otherwise $y \sim 2r$, a contradiction.  
 Suppose $y \sim a$. Then $y \sim a \sim x $ and so
$d(x, y) \leq 2$.
Suppose $y \nsim a$. Then $[y, a] = r$ or $-r$. If $[y, a] = r$ then  
$[y, r + a] =   0$. 
 Therefore, $y \sim r + a \sim a \sim x$. 
Therefore, $d(x, y) \leq 3$.
If $[y, a] = -r$ then  $[y, 2r + a] =  0$.  Therefore, $y \sim 2r + a \sim a \sim x$ and so $d(x, y) \leq 3$. Hence, $\diam(\Delta_R^r) \leq 3$.

\noindent \textbf{Subcase 2.2:} $x \nsim a$

In this case we have $x \nsim r$ and $x \nsim 2r$. It can be seen that $[x, r] = r$ implies  $[x, 2r] =  -r$ and $[x, r] = -r$ implies  $[x, 2r] =  r$.  

Suppose $[x, r] = r$ and $[x, a] = r$. Then $[x, r - a] = [x, r] - [x, a] = 0$. Hence, $x \sim r - a$. 
Now, we have the following cases.
\begin{enumerate}
\item[(i)] $x \sim r - a \sim r \sim y$ if $y \sim r$. 

\item[(ii)] $x \sim r - a \sim 2r \sim y$ if $y \nsim r$ but $y \sim 2r$ .
\end{enumerate}

\noindent Suppose $y \nsim r$ as well as $y \nsim 2r$. 
Then, proceeding as in Subcase 2.1, we get the following cases:
\begin{enumerate}
\item[(iii)] $x \sim r - a \sim a \sim y$ if  $y \nsim r$ and $2r$ but  $y \sim a$.
\item[(iv)] $y \sim r - a \sim x$ if $[y, r] = r$ and $[y, a] = r$.
\item[(v)] $y \sim 2r - a \sim r - a \sim x$ if $[y, r] = r$ and $[y, a] = -r$.
\item[(vi)] $y \sim r + a \sim r - a \sim x$ if $[y, r] = -r$ and $[y, a] = r$.
\item[(vii)] $y \sim 2r + a \sim r - a \sim x$ if $[y, r] = -r$ and $[y, a] = -r$.
\end{enumerate}
Therefore, $d(x, y) \leq 3$.

Suppose $[x, r] = r$ and $[x, a] = -r$. Then 
\[
[x, 2r - a] = [x, 2r] - [x, a] = -r - (-r) = 0.
\]
Hence, $x \sim 2r - a$. 
Now, proceeding as above we get the  following cases:
\begin{enumerate}
\item[(i)] $x \sim 2r - a \sim r \sim y$ if $y \sim r$.

\item[(ii)] $x \sim 2r - a \sim 2r \sim y$ if $y \nsim r$ but $y \sim 2r$.

\item[(iii)] $x \sim 2r - a \sim a \sim y$ if $y \nsim r$ and $2r$ but  $y \sim a$.

\item[(iv)] $y \sim r - a \sim 2r - a \sim x$ if $[y, r] = r$ and $[y, a] = r$.

\item[(v)] $y \sim 2r - a \sim  x$ if $[y, r] = r$ and $[y, a] = -r$.

\item[(vi)] $y \sim r + a \sim 2r - a \sim x$ if $[y, r] = -r$ and $[y, a] = r$.

\item[(vii)] $y \sim 2r + a \sim 2r - a \sim x$ if $[y, r] = -r$ and $[y, a] = -r$.
\end{enumerate}
Therefore, $d(x, y) \leq 3$.

Suppose $[x, r] = -r$ and $[x, a] = r$. Then
 \[
[x, r + a] = [x, r] + [x, a] = -r + r = 0.
\]
Hence, $x \sim r + a$. 
Proceeding as above we get the  following similar cases:
\begin{enumerate}
\item[(i)] $x \sim r + a \sim r \sim y$ if $y \sim r$.

\item[(ii)] $x \sim r + a \sim 2r \sim y$ if $y \nsim r$ but $y \sim 2r$.

\item[(iii)] $x \sim r + a \sim a \sim y$ if $y \nsim r$ and $2r$ but  $y \sim a$.

\item[(iv)] $y \sim r - a \sim r + a \sim x$ if $[y, r] = r$ and $[y, a] = r$.

\item[(v)] $y \sim 2r - a \sim r + a \sim  x$ if $[y, r] = r$ and $[y, a] = -r$.

\item[(vi)] $y \sim r + a \sim  x$ if $[y, r] = -r$ and $[y, a] = r$.

\item[(vii)] $y \sim 2r + a \sim r + a \sim x$ if $[y, r] = -r$ and $[y, a] = -r$.
\end{enumerate}
Therefore, $d(x, y) \leq 3$.

Suppose $[x, r] = -r$ and $[x, a] = -r$. Then 
\[
[x, 2r + a] = [x, 2r] + [x, a] = r + (-r) = 0.
\]
Hence, $x \sim 2r + a$ and so we get the the following similar cases:
\begin{enumerate}
\item[(i)] $x \sim 2r + a \sim r \sim y$ if $y \sim r$.

\item[(ii)] $x \sim 2r + a \sim 2r \sim y$ if $y \nsim r$ but $y \sim 2r$.

\item[(iii)] $x \sim 2r + a \sim a \sim y$ if $y \nsim r$ and $2r$ but  $y \sim a$.

\item[(iv)] $y \sim r - a \sim 2r + a \sim x$ if $[y, r] = r$ and $[y, a] = r$.

\item[(v)] $y \sim 2r - a \sim 2r + a \sim  x$ if $[y, r] = r$ and $[y, a] = -r$.

\item[(vi)] $y \sim r + a \sim 2r + a \sim x$ if $[y, r] = -r$ and $[y, a] = r$.

\item[(vii)] $y \sim 2r + a  \sim x$ if $[y, r] = -r$ and $[y, a] = -r$.
\end{enumerate}
Therefore, $d(x, y) \leq 3$. Hence, in all the cases $\diam(\Delta_R^r) \leq 3$. This completes the proof
\end{proof}

As a consequence of Proposition \ref{deg_prop_1}(a) and Corollary \ref{deg_cor_1} we get the following result.
\begin{proposition}\label{induceddeg_cor_1}
Let $x$ be any vertex in  $\Delta_R^r$.
\begin{enumerate}
\item If $r = 0$ then  $\deg(x)= |R| - |C_R(x)|$.

\item If $r \neq 0$ and  $2r= 0$ then
$$\deg(x) = \begin{cases}
           |R| - |Z(R)| - 1,  &\mbox{if $T_{x, r} = \emptyset$}  \\
           |R| - |Z(R)| - |C_R(x)| - 1, & \mbox{otherwise}.
\end{cases}$$
\item If $r \neq 0$ and  $2r \neq 0$ then
$$\deg(x) = \begin{cases}
           |R| - |Z(R)| - 1, & \mbox{if $T_{x, r} = \emptyset$}  \\
           |R| - |Z(R)| - 2|C_R(x)| - 1, & \mbox{otherwise}.
\end{cases}$$
\end{enumerate}
\end{proposition}
\noindent Some applications of Proposition \ref{induceddeg_cor_1} are given below.
\begin{theorem} \label{not-tree-2}
  Let $R$ be a noncommutative ring such that $|R| \ne 8$.  Then $\Delta_R^r$ is not a tree.
\end{theorem}
\begin{proof}
Suppose that $\Delta_R^r$  is a tree. Therefore there exist $x\in R \setminus Z(R)$ such that $\deg(x) = 1$.

\noindent\textbf{ Case 1:} $r = 0$. 

By Proposition \ref{induceddeg_cor_1}(a), we have $\deg(x) = |R| - |C_R(x)|$. Therefore,  $|R| - |C_R(x)| = 1$ and hence $|C_R(x)| = 1$, a contradiction. 


\noindent\textbf{ Case 2:} $r\neq 0$ and $2r = 0$.

  By Proposition \ref{induceddeg_cor_1}(b), we have  $\deg(x) = |R| - |Z(R)| - 1$ or $|R| -  |Z(R)| - |C_R(x)| - 1$. Hence $|R| - |Z(R)| - 1 = 1$ or $|R|- |Z(R)| - |C_R(x)| - 1 = 1$. 

\noindent \textbf{Subcase 2.1:}  $|R| - |Z(R)| = 2$.

 In this case we have $|Z(R)| = 1$ or $2$. If $|Z(R)| = 1$ then $|R| = 3$, a contradiction. If $|Z(R)| = 2$ then $|R| = 4$. Therefore, the additive quotient group    $\frac{R}{Z(R)}$ is    cyclic. Hence, $R$ is commutative; a contradiction.  
   
\noindent \textbf{Subcase 2.2:} $|R|- |Z(R)| - |C_R(x)|   = 2$.

 In this case,  $|Z(R)| = 1$ or $2$. If $|Z(R)| = 1$ then $|R| - |C_R(x)| = 3$. Therefore, $|C_R(x)| = 3$ and hence $|R| = 6$. Therefore, $R$ is commutative; a contradiction.  If $|Z(R)| = 2$ then $|R| - |C_R(x)|   = 4$. Therefore, $|C_R(x)| = 4$
 and so $|R| = 8$, a contradiction. 
   

\noindent \textbf{Case 3:} $r \neq 0$ and  $2r\neq 0$.

  By Proposition \ref{induceddeg_cor_1}(c), we have $\deg(x) = |R| - |Z(R)| - 1$ or $|R|- |Z(R)| - 2|C_R(x)| -1$. Hence,  $|R| - |Z(R)| - 1 =1$ or $|R| - |Z(R)| - 2|C_R(x)| -1 = 1$. If $|R| - |Z(R)| = 2$ then as shown in subcase 2.1 we get a contradiction. If $|R| - |Z(R)| - 2|C_R(x)|   = 2$ then $|Z(R)| = 1$ or $2$.
   
\noindent \textbf{Subcase 3.1:}  $|Z(R)| = 1$.

 In this case, $|R| - 2|C_R(x)| = 3$. Therefore, $|C_R(x)| = 3$ and hence $|R| = 9$. It follows from Fig. $4$ that $\Delta^r_R = 4K_2$  which is a contradiction.

\noindent \textbf{Subcase 3.2:} $|Z(R)| = 2$.

 In this case, $|R| - 2|C_R(x)| = 4$. Therefore, $|C_R(x)| = 4$
 and so  $|R| = 12$. It follows that the additive quotient group $\frac{R}{Z(R)}$ is cyclic.  Hence, $R$ is commutative; a contradiction.    This completes the proof. 
%
%
\end{proof}
The proof of the above theorem also gives the following results.
\begin{theorem}  
Let $R$ be a noncommutative ring such that $|R| \ne 8$.  Then $\Delta_R^r$ has end vertices if and only if $r \ne 0$ and $R$ is isomorphic to $E(9)$ or $F(9)$.
\end{theorem}

\begin{theorem}  
Let $R$ be a noncommutative ring such that $|R| \ne 8$.  Then $\Delta_R^r$ is $1$-regular if and only if $r \ne 0$ and $R$ is isomorphic to $E(9)$ or $F(9)$.
\end{theorem}

We also have the following corollary.
\begin{corollary}
Let $R$ be a noncommutative ring such that $|R| \ne 8$. Then the noncommuting graph of $R$   is not a tree. Further, noncommuting graph of such rings do  not have any end vertices.
\end{corollary}

\begin{theorem}  
Let $R$ be a noncommutative ring such that $|R| \ne  8, 12$.  Then $\Delta_R^r$ has a   vertex of degree $2$ if and only if $r = 0$ and $R$ is isomorphic to $E(4)$ or $F(4)$.
\end{theorem}
\begin{proof}
Suppose $\Delta_R^r$ has a vertex $x$ of degree $2$.

\noindent \textbf{Case 1:}  $r = 0$.

By Proposition \ref{induceddeg_cor_1}(a), we have $\deg(x) = |R| - |C_R(x)|$. Therefore,  $|R| - |C_R(x)| = 2$ and hence $|C_R(x)| = 2$. Therefore, $|R| = 4$ and  $\Delta_R^r$ is a triangle (as shown in Figure 1).   
%

\noindent \textbf{Case 2:}   $r \neq 0$ and $2r = 0$.

By Proposition \ref{induceddeg_cor_1}(b), we have  $deg(x) = |R| - |Z(R)| - 1$ or 
  $deg(x) = |R| - |Z(R)| - |C_R(x)| - 1$. Therefore  $|R| - |Z(R)| - 1 = 2$ or 
  $|R| - |Z(R)| - |C_R(x)| - 1 = 2$.
 
\noindent \textbf{Subcase 2.1:}  $|R| - |Z(R)| = 3$.

In this case we have $|Z(R)| = 1$ or $3$. If $|Z(R)| = 1$ then $|R| = 4$. As shown in Figure 2, $\Delta_R^r$ is a $0$-regular graph on three vertices. Therefore, it has no vertex of degree $2$, which is a contradiction. If $|Z(R)| = 3$ then $|R| = 6$. Therefore, $R$ is commutative; a contradiction. 

\noindent \textbf{Subcase 2.2:}  $|R| - |Z(R)| - |C_R(x)| = 3$.
 
 In this case,  $|Z(R)| = 1$ or $3$. If $|Z(R)| = 1$ then $|R| - |C_R(x)| = 4$. Therefore, $|C_R(x)| = 2$ or $4$ and hence $|R| = 6$ or $8$; a contradiction.
   If $|Z(R)| = 3$ then $|R| - |C_R(x)|   = 6$. Therefore, $|C_R(x)| = 6$
 and so $|R| = 12$, which contradicts our assumption.

\noindent \textbf{Case 3:} $r \neq 0$ and  $2r\neq 0$.

  By Proposition \ref{induceddeg_cor_1}(c), we have $\deg(x) = |R| - |Z(R)| - 1$ or $|R|- |Z(R)| - 2|C_R(x)| -1$. Hence,  $|R| - |Z(R)| - 1 = 2$ or $|R| - |Z(R)| - 2|C_R(x)| -1 = 2$. 
  
If $|R| - |Z(R)| = 3$ then as shown in Subcase 2.1 we get a contradiction. If $|R| - |Z(R)| - 2|C_R(x)|   = 3$ then $|Z(R)| = 1$ or $3$.
   
\noindent \textbf{Subcase 3.1:}  $|Z(R)| = 1$.

 In this case, $|R| - 2|C_R(x)| = 4$. Therefore, $|C_R(x)| = 2$ or $4$ and hence $|R| = 8$ or $12$ which is a contradiction.

\noindent \textbf{Subcase 3.2:} $|Z(R)| = 3$.

 In this case, $|R| - 2|C_R(x)| = 6$. Therefore, $|C_R(x)| = 6$
 and so $|R| = 18$.  It follows that the additive quotient group $\frac{R}{Z(R)}$ is cyclic.  Hence, $R$ is commutative; a contradiction. 
 This completes the proof.
\end{proof}
The proof of the above result also suggest the following theorem.
\begin{theorem}  
Let $R$ be a noncommutative ring such that $|R| \ne  8, 12$.  Then $\Delta_R^r$ is $2$-regular if and only if $r = 0$ and $R$ is isomorphic to $E(4)$ or $F(4)$.
\end{theorem}
\begin{corollary}  
Let $R$ be a noncommutative ring such that $|R| \ne  8, 12$.  Then the noncommuting graph of $R$ is  $2$-regular if and only if  $R$ is isomorphic to $E(4)$ or $F(4)$.
\end{corollary}

\begin{theorem}  
Let $R$ be a noncommutative ring such that $|R| \ne   16, 18$.  Then $\Delta_R^r$ has no vertex of degree $3$.
\end{theorem}
\begin{proof}
	Suppose $\Delta_R^r$ has a vertex $x$ of degree $3$.
	
	\noindent \textbf{Case 1:}  $r = 0$.
	
	By Proposition \ref{induceddeg_cor_1}(a), we have $\deg(x) = |R| - |C_R(x)|$. Therefore,  $|R| - |C_R(x)| = 3$ and hence $|C_R(x)| = 3$. Therefore, $|R| = 6$ and hence $R$ is commutative; a contradiction.   
	
	\noindent \textbf{Case 2:}   $r \neq 0$ and $2r = 0$.
	
	By Proposition \ref{induceddeg_cor_1}(b), we have  $deg(x) = |R| - |Z(R)| - 1$ or $deg(x) = |R| - |Z(R)| - |C_R(x)| - 1$. Therefore  $|R| - |Z(R)| - 1 = 3$ or $|R| - |Z(R)| - |C_R(x)| - 1 = 3$.
	
	\noindent \textbf{Subcase 2.1:}  $|R| - |Z(R)| = 4$.
	
	In this case we have $|Z(R)| = 1$ or $2$ or $4$. If $|Z(R)| = 1$ or $2$ then $|R| = 5$ or $6$ and hence $R$ is commutative; a contradiction. If $|Z(R)| = 4$ then $|R| = 8$. Therefore, the additive quotient group $\frac{R}{Z(R)}$ is cyclic.  Hence, $R$ is commutative; a contradiction. 
	
	\noindent \textbf{Subcase 2.2:}  $|R| - |Z(R)| - |C_R(x)| = 4$.
	
	In this case,  $|Z(R)| = 1$ or $2$ or $4$. If $|Z(R)| = 1$ then $|R| - |C_R(x)| = 5$. Therefore, $|C_R(x)| = 5$ and hence $|R| = 10$. Therefore $R$ is commutative; a contradiction.
	If $|Z(R)| = 2$ then $|R| - |C_R(x)|   = 6$. Therefore, $|C_R(x)| = 6$
and so $|R| = 12$.  It follows that the additive quotient group $\frac{R}{Z(R)}$ is cyclic.  Hence, $R$ is commutative; a contradiction.
	If $|Z(R)| = 4$ then $|R| - |C_R(x)|   = 8$. Therefore, $|C_R(x)| = 8$
and so $|R| = 16$; a contradiction.

	\noindent \textbf{Case 3:} $r \neq 0$ and  $2r\neq 0$.
	
	By Proposition \ref{induceddeg_cor_1}(c), we have $\deg(x) = |R| - |Z(R)| - 1$ or $|R|- |Z(R)| - 2|C_R(x)| -1$. Hence,  $|R| - |Z(R)| - 1 = 3$ or $|R| - |Z(R)| - 2|C_R(x)| -1 = 3$. 
	
	If $|R| - |Z(R)| = 4$ then as shown in Subcase 2.1 we get a contradiction. If $|R| - |Z(R)| - 2|C_R(x)|   = 4$ then $|Z(R)| = 1$ or $2$ or $4$.
	
	\noindent \textbf{Subcase 3.1:}  $|Z(R)| = 1$.
	
	In this case, $|R| - 2|C_R(x)| = 5$. Therefore, $|C_R(x)| = 5$ then $|R| = 15$. Therefore $R$ is commutative; a contradiction. 
	
	\noindent \textbf{Subcase 3.2:} $|Z(R)| = 2$.
	
	In this case, $|R| - 2|C_R(x)| = 6$. Therefore, $|C_R(x)| = 6$
and so $|R| = 18$; a contradiction.
	
	\noindent \textbf{Subcase 3.3:} $|Z(R)| = 4$.
	
	In this case, $|R| - 2|C_R(x)| = 8$. Therefore, $|C_R(x)| = 8$
and so $|R| = 24$. It follows that the additive quotient group $\frac{R}{Z(R)}$ is cyclic.  Hence, $R$ is commutative; a contradiction. 
		This completes the proof.
\end{proof}

\begin{corollary}
Let $R$ be a noncommutative ring such that $|R| \ne  16, 18$.  Then $\Delta_R^r$  is not $3$-regular. In particular, the noncommuting graph of such  $R$  is not $3$-regular.
\end{corollary}

\begin{theorem}  
Let $R$ be a noncommutative ring such that $|R| \ne  8, 12, 18, 20$.  Then $\Delta_R^r$ has no vertex of degree $4$.
\end{theorem}
\begin{proof}
	Suppose $\Delta_R^r$ has a vertex $x$ of degree $4$.
	
	\noindent \textbf{Case 1:}  $r = 0$.
	
	By Proposition \ref{induceddeg_cor_1}(a), we have $\deg(x) = |R| - |C_R(x)|$. Therefore,  $|R| - |C_R(x)| = 4$ and hence $|C_R(x)| = 2$ or $4$. If $|C_R(x)| = 2$ then $|R| = 6$ and hence $R$ is commutative; a contradiction. If $|C_R(x)| = 4$ then $|R| = 8$; a contradiction.   
	
	\noindent \textbf{Case 2:}   $r \neq 0$ and $2r = 0$.
	
	By Proposition \ref{induceddeg_cor_1}(b), we have  $deg(x) = |R| - |Z(R)| - 1$ or $deg(x) = |R| - |Z(R)| - |C_R(x)| - 1$. Therefore  $|R| - |Z(R)| - 1 = 4$ or $|R| - |Z(R)| - |C_R(x)| - 1 = 4$.
	
	\noindent \textbf{Subcase 2.1:}  $|R| - |Z(R)| = 5$.
	
	In this case we have $|Z(R)| = 1$ or $5$. Then $|R| = 6$ or $10$ and hence $R$ is commutative; a contradiction.
	
	\noindent \textbf{Subcase 2.2:}  $|R| - |Z(R)| - |C_R(x)| = 5$.
	
	In this case,  $|Z(R)| = 1$ or $5$. If $|Z(R)| = 1$ then $|R| - |C_R(x)| = 6$. Therefore, $|C_R(x)| = 2$ or $3$ or $6$. If  $|C_R(x)| = 2$ then $|R| = 8$; a contradiction. If $|C_R(x)| = 3$ then $|R| = 9$. It follows from Figure $4$ that $\Delta^r_R = 4K_2$  which is a contradiction. If $|C_R(x)| = 6$ then $|R| = 12$; a contradiction.
	If $|Z(R)| = 5$ then $|R| - |C_R(x)| = 10$.  Therefore, $|C_R(x)| = 10$	
	and so $|R| = 20$; a contradiction.

	\noindent \textbf{Case 3:} $r \neq 0$ and  $2r\neq 0$.
	
	By Proposition \ref{induceddeg_cor_1}(c), we have $\deg(x) = |R| - |Z(R)| - 1$ or $|R|- |Z(R)| - 2|C_R(x)| -1$. Hence,  $|R| - |Z(R)| - 1 = 4$ or $|R| - |Z(R)| - 2|C_R(x)| -1 = 4$. 
	
	If $|R| - |Z(R)| = 5$ then as shown in Subcase 2.1 we get a contradiction. If $|R| - |Z(R)| - 2|C_R(x)|   = 5$ then $|Z(R)| = 1$ or $5$.
	
	\noindent \textbf{Subcase 3.1:}  $|Z(R)| = 1$.
	
	In this case, $|R| - 2|C_R(x)| = 6$. Therefore, $|C_R(x)| = 2$ or $3$ or $6$. If $|C_R(x)| = 2$ then $|R| = 10$. Therefore $R$ is commutative; a contradiction. If $|C_R(x)| = 3$ or $6$ then $|R| = 12$ or $18$; a contradiction. 
	
	\noindent \textbf{Subcase 3.2:} $|Z(R)| = 5$.
	
	In this case, $|R| - 2|C_R(x)| = 10$. Therefore, $|C_R(x)| = 10$
	and so $|R| = 30$. It follows that the additive quotient group $\frac{R}{Z(R)}$ is cyclic.  Hence, $R$ is commutative; a contradiction. 
	This completes the proof.
\end{proof}

\begin{corollary}
Let $R$ be a noncommutative ring such that $|R| \ne   8, 12, 18, 20$.  Then $\Delta_R^r$  is not $4$-regular. In particular, the noncommuting graph of such  $R$  is not $4$-regular.
\end{corollary}

\begin{theorem}  
Let $R$ be a noncommutative ring such that $|R| \ne  8,  16, 24, 27$.  Then $\Delta_R^r$ has no vertex of degree $5$.
\end{theorem}
\begin{proof}
	Suppose $\Delta_R^r$ has a vertex $x$ of degree $5$.
	
	\noindent \textbf{Case 1:}  $r = 0$.
	
	By Proposition \ref{induceddeg_cor_1}(a), we have $\deg(x) = |R| - |C_R(x)|$. Therefore,  $|R| - |C_R(x)| = 5$ and hence $|C_R(x)| = 5$. Then $|R| = 10$ and hence $R$ is commutative; a contradiction.   
	
	\noindent \textbf{Case 2:}   $r \neq 0$ and $2r = 0$.
	
	By Proposition \ref{induceddeg_cor_1}(b), we have  $deg(x) = |R| - |Z(R)| - 1$ or $deg(x) = |R| - |Z(R)| - |C_R(x)| - 1$. Therefore  $|R| - |Z(R)| - 1 = 5$ or $|R| - |Z(R)| - |C_R(x)| - 1 = 5$.
	
	\noindent \textbf{Subcase 2.1:}  $|R| - |Z(R)| = 6$.
	
	In this case we have $|Z(R)| = 1$ or $2$ or $3$ or $6$. If  $|Z(R)| = 1$ then $|R| = 7$ and hence $R$ is commutative; a contradiction. 
	If  $|Z(R)| = 2$ then $|R| = 8$; a contradiction. 
	If  $|Z(R)| = 3$ then $|R| = 9$. It follows from Figure $4$ that $\Delta^r_R = 4K_2$  which is a contradiction. 
	If  $|Z(R)| = 6$ then $|R| = 12$. Therefore, the additive quotient group $\frac{R}{Z(R)}$ is cyclic.  Hence, $R$ is commutative; a contradiction.
	
	\noindent \textbf{Subcase 2.2:}  $|R| - |Z(R)| - |C_R(x)| = 6$.
	
	In this case,  $|Z(R)| = 1$ or $2$ or $3$ or $6$. If $|Z(R)| = 1$ then $|R| - |C_R(x)| = 7$. Therefore, $|C_R(x)| = 7$ then $|R| = 14$ and hence $R$ is commutative; a contradiction. 
	If $|Z(R)| = 2$ then $|R| - |C_R(x)| = 8$.  Therefore, $|C_R(x)| =  4$ or $8$. 
	If $|C_R(x)| = 4$ then $|R| = 12$. Therefore, the additive quotient group $\frac{R}{Z(R)}$ is cyclic.  Hence, $R$ is commutative; a contradiction.
	If $|C_R(x)| = 8$ then  $|R| = 16$; a contradiction. If $|Z(R)| = 3$ then $|R| - |C_R(x)| = 9$.  Therefore, $|C_R(x)| = 9$. 
	and so $|R| = 18$. It follows that the additive quotient group $\frac{R}{Z(R)}$ is cyclic.  Hence, $R$ is commutative; a contradiction.  If $|Z(R)| = 6$ then $|R| - |C_R(x)| = 12$.  Therefore, $|C_R(x)| = 12$ 
	and so $|R| = 24$; a contradiction.

	\noindent \textbf{Case 3:} $r \neq 0$ and  $2r\neq 0$.
	
	By Proposition \ref{induceddeg_cor_1}(c), we have $\deg(x) = |R| - |Z(R)| - 1$ or $|R|- |Z(R)| - 2|C_R(x)| -1$. Hence,  $|R| - |Z(R)| - 1 = 5$ or $|R| - |Z(R)| - 2|C_R(x)| -1 = 5$. 
	
	If $|R| - |Z(R)| = 6$ then as shown in Subcase 2.1 we get a contradiction. If $|R| - |Z(R)| - 2|C_R(x)|   = 6$ then $|Z(R)| = 1$  or $2$ or $3$ or $6$.
	
	\noindent \textbf{Subcase 3.1:}  $|Z(R)| = 1$.
	
	Here we have, $|R| - 2|C_R(x)| = 7$. Therefore, $|C_R(x)| = 7$ then $|R| = 21$ and hence $R$ is commutative; a contradiction.
	
	\noindent \textbf{Subcase 3.2:} $|Z(R)| = 2$.
	
	In this case, $|R| - 2|C_R(x)| = 8$. Therefore, $|C_R(x)| = 4$ or $8$. 
	If $|C_R(x)| = 4$ or $8$  then  $|R| = 16$ or $24$; a contradiction.
	
	\noindent \textbf{Subcase 3.3:} $|Z(R)| = 3$.
	
	In this case, $|R| - 2|C_R(x)| = 9$. Therefore, $|C_R(x)| = 9$
	and so $|R| = 27$; a contradiction.
	
	\noindent \textbf{Subcase 3.4:} $|Z(R)| = 6$.
	
	In this case, $|R| - 2|C_R(x)| = 12$. Therefore, $|C_R(x)| = 12$
and so $|R| = 36$. It follows that the additive quotient group $\frac{R}{Z(R)}$ is cyclic.  Hence, $R$ is commutative; a contradiction.
	This completes the proof.
\end{proof}

\begin{corollary}
Let $R$ be a noncommutative ring such that $|R| \ne   8,  16, 24,27$.  Then $\Delta_R^r$  is not $5$-regular. In particular, the noncommuting graph of such  $R$  is not $5$-regular.
\end{corollary}

We conclude this section with the following characterization of $R$.
\begin{theorem}  
Let $R$ be a noncommutative ring such that $|R| \ne  8, 12, 16$, $24, 28$.  Then $\Delta_R^r$ has a vertex of degree $6$ if and only if  $r=0$ and $R$ is isomorphic to $E(9)$ or $F(9)$. 
\end{theorem}
\begin{proof}
	Suppose $\Delta_R^r$ has a vertex $x$ of degree $6$.
	
	\noindent \textbf{Case 1:}  $r = 0$.
	
	By Proposition \ref{induceddeg_cor_1}(a), we have $\deg(x) = |R| - |C_R(x)|$. Therefore,  $|R| - |C_R(x)| = 6$ and hence $|C_R(x)| = 2$ or $3$ or $6$. If $|C_R(x)| = 2$ then $|R| = 8$; a contradiction. If $|C_R(x)| = 3$ then $|R| = 9$. Therefore, $\Delta_R^r$ is a $6$-regular graph (as shown in Figure 3).    If $|C_R(x)| = 6$ then $|R| = 12$; a contradiction.   
	
	\noindent \textbf{Case 2:}   $r \neq 0$ and $2r = 0$.
	
	By Proposition \ref{induceddeg_cor_1}(b), we have  $deg(x) = |R| - |Z(R)| - 1$ or $deg(x) = |R| - |Z(R)| - |C_R(x)| - 1$. Therefore  $|R| - |Z(R)| - 1 = 6$ or $|R| - |Z(R)| - |C_R(x)| - 1 = 6$.
	
	\noindent \textbf{Subcase 2.1:}  $|R| - |Z(R)| = 7$.
	
	In this case we have $|Z(R)| = 1$ or $7$. If  $|Z(R)| = 1$ then $|R| = 8$; a contradiction. If  $|Z(R)| = 7$ then $|R| = 14$ and hence $R$ is commutative; a contradiction.
	
	\noindent \textbf{Subcase 2.2:}  $|R| - |Z(R)| - |C_R(x)| = 7$.
	
	In this case,  $|Z(R)| = 1$ or $7$. If $|Z(R)| = 1$ then $|R| - |C_R(x)| = 8$. Therefore, $|C_R(x)| = 2$ or $4$ or $8$. If  $|C_R(x)| = 2$ then $|R| = 10$. Thus $R$ is commutative; a contradiction. If $|C_R(x)| = 4$ or $8$ then $|R| = 12$ or $16$; which are contradictions. 
	If $|Z(R)| = 7$ then $|R| - |C_R(x)| = 14$.  Therefore, $|C_R(x)| = 14$
and so $|R| = 28$; a contradiction.

	\noindent \textbf{Case 3:} $r \neq 0$ and  $2r\neq 0$.
	
	By Proposition \ref{induceddeg_cor_1}(c), we have $\deg(x) = |R| - |Z(R)| - 1$ or $|R|- |Z(R)| - 2|C_R(x)| -1$. Hence,  $|R| - |Z(R)| - 1 = 6$ or $|R| - |Z(R)| - 2|C_R(x)| -1 = 6$. 
	
	If $|R| - |Z(R)| = 7$ then as shown in Subcase 2.1 we get a contradiction. If $|R| - |Z(R)| - 2|C_R(x)|   = 7$ then $|Z(R)| = 1$ or $7$.
	
	\noindent \textbf{Subcase 3.1:}  $|Z(R)| = 1$.
	
	In this case, $|R| - 2|C_R(x)| = 8$. Therefore, $|C_R(x)| = 2$ or $4$ or $8$ then $|R| = 12$ or $16$ or $24$; all are contradictions to the order of $R$.
	
	\noindent \textbf{Subcase 3.2:} $|Z(R)| = 7$.
	
	In this case, $|R| - 2|C_R(x)| = 14$. Therefore, $|C_R(x)| = 14$
and so	$|R| = 42$. It follows that the additive quotient group $\frac{R}{Z(R)}$ is cyclic.  Hence, $R$ is commutative; a contradiction. 
		This completes the proof.
\end{proof}

\begin{corollary}  
Let $R$ be a noncommutative ring such that $|R| \ne  8, 12, 16$, $24, 28$.  Then $\Delta_R^r$ is $6$-regular if and only if  $r=0$ and $R$ is isomorphic to $E(9)$ or $F(9)$. In particular, the noncommuting graph of such  $R$  is $6$-regular if and only if    $R$ is isomorphic to $E(9)$ or $F(9)$.
\end{corollary}

\section*{Acknowledgment}
M. Sharma would like to thank  DST for the INSPIRE Fellowship.


\begin{thebibliography}{3}





\bibitem{BMS}
S. M. Buckley, D. Machale, and A. N$\acute{\rm i}$ Sh$\acute{\rm e}$,  \emph{Finite rings with many commuting pairs of elements}, Preprint.





\bibitem{jutirekha2}
J. Dutta, D. K. Basnet and R. K. Nath, {\em On generalized non-commuting graph of a finite ring}, Algebra Colloq., {\bf 25}(1) (2018), 1149--160.

\bibitem{dutta-nath}
J. Dutta, W. N. T. Fasfous and R. K. Nath, {\em Spectrum and genus of commuting graphs of some classes of finite rings}, Acta Comment. Univ. Tartu. Math., {\bf  23}(1) (2019),  5--12.
\bibitem{erfanian1}
A. Erfanian, K. Khashyarmanesh and Kh. Nafar, \emph{Non-commuting graphs of rings}, Discrete Math. Algorithms Appl., {\bf 7}(3) (2015), 1550027 (7 pages).

\bibitem{fine93}
B. Fine,    {\em  Classification of finite rings of order $p^2$},   Math. Mag., {\bf 66} (1993), 248--252.

\bibitem{pH40}
P. Hall,    {\em The classification of prime power groups},  J. Reine Angew. Math., {\bf 182} (1940), 130--141.



\bibitem{nEGJ-16}  M. Nasiri, A. Erfanian, M. Ganjali and A. Jafarzadeh, \emph{$g$-noncommuting graphs of finite groups}, J. Prime Research  Math. \textbf{12} (2016), 16-23.

\bibitem{nEGJ-17} M. Nasiri, A. Erfanian, M. Ganjali and A. Jafarzadeh, \emph{Isomorphic $g$-noncommuting graphs of finite groups}, Publ. Math. Debrecen \textbf{91(1-2)} (2017), 33-42.

\bibitem{nEM-18} M. Nasiri, A. Erfanian and A. Mohammadian, \emph{Connectivity and planarity of $g$-noncommuting graphs of finite groups}, J. Agebra  Appl. \textbf{16(2)} (2018), 1850107 (9 pages). 


\bibitem{o&v}
G. R. Omidi and E. Vatandoost,  \emph{On the commuting graph of rings},  J. Algebra Appl., {\bf 10} (3) (2011), 521--527.



\bibitem{tolue}
B. Tolue, A. Erfanian and A. Jafarzadeh, \emph{A kind of non-commuting graph of finite groups}, J. Sci. I. R. Iran, {\bf 25} (4) (2014), 379-384.




\end{thebibliography}
\end{document}